\documentclass[12pt]{article}%
\usepackage{amssymb}
\usepackage{amsfonts}
\usepackage{graphicx, epsfig}
\usepackage{amsmath}
\usepackage{amsthm}
\usepackage{morefloats}
\usepackage{placeins}
\usepackage{rotating}
\usepackage{adjustbox}
\usepackage{authblk}
\usepackage[left=1cm, right=1cm]{geometry}
\usepackage{booktabs}

\newcommand{\head}[1]{\textnormal{\textbf{#1}}}

\RequirePackage[colorlinks,citecolor=blue,urlcolor=blue]{hyperref}

\newtheorem{theorem}{Theorem}[section]
\newtheorem{corollary}[theorem]{Corollary}
\newtheorem{lemma}[theorem]{Lemma}
\newtheorem{proposition}[theorem]{Proposition}

\newtheorem{definition}[theorem]{Definition}

\newtheorem{remark}[theorem]{Remark}

\newtheorem{condition}[theorem]{Condition}

\def\OLS{\textsc{ols}}
\def\RRR{\textsc{RRR}}

\newcommand{\real}{\mathbb R}
\newcommand{\spn}{\mathrm{span}}
\newcommand{\diag}{\mathrm{diag}}
\newcommand{\E}{\mathrm{E}}
\newcommand{\var}{\mathrm{var}}

\newcommand{\pr}{\mathrm{pr}}

\newcommand{\rank}{\mathrm{rank}}

\newcommand{\vecc}{\mathrm{vec}}
\newcommand{\MMn}{M_n}
\newcommand{\MM}{M}
\newcommand{\res}{{\text{res}}}
\newcommand{\firsts}{{\scriptstyle \first}}
\newcommand{\first}{\mathrm{first}}
\newcommand{\sxires}{\widehat{{\scriptstyle \xi}}_n^{\scriptstyle{\text{res}}}}

\newcommand{\Tan}{T_{\xis_0}}
\newcommand{\xires}{\widehat{\xi}_n^{\res}}
\newcommand{\xis}{{\scriptstyle \xi}}
\newcommand{\betas}{\scriptstyle \beta}

\newcommand{\etas}{\scriptstyle \eta}

\begin{document}

\title{Asymptotic theory for maximum likelihood estimates in reduced-rank multivariate generalised linear models}

\author[a$^\ast$]{Efstathia Bura \thanks{Corresponding author. Email: efstathia.bura@tuwien.ac.at}}
\author[b]{Sabrina Duarte}
\author[b]{Liliana Forzani}
\author[c]{Ezequiel Smucler}
\author[c]{Mariela Sued}

\affil[a$^\ast$]{Institute of Statistics and Mathematical Methods in Economics, TU Wien, A-1040 Vienna, Austria and Department of Statistics, George Washington University, Washington, D.C. 20052, U.S.A.}
\affil[b]{Facultad de Ingenier\'{i}a Qu\'{i}mica, UNL, Researcher of CONICET, 3000 Santa Fe, Argentina}
\affil[c]{Instituto de C\'{a}lculo, UBA,  CONICET, FCEN,  1428 Buenos Aires, Argentina}

\date{}

\maketitle

\begin{abstract}
Reduced-rank regression is a dimensionality reduction method with many applications. The asymptotic theory for reduced rank estimators of parameter matrices in multivariate linear models has been studied extensively.
In contrast, few  theoretical results are available for reduced-rank multivariate generalised linear models. We develop M-estimation theory for concave criterion functions that are maximised over parameters spaces that are neither convex nor closed. These results are used to  derive the consistency and asymptotic  distribution of maximum likelihood estimators in  reduced-rank multivariate generalised linear models, when the response and predictor vectors
have a joint distribution. We illustrate our results  in a real data classification problem with binary covariates.
\end{abstract}

\noindent%
{\it Keywords:} M-estimation; exponential family; rank restriction; non-convex; parameter spaces
\vfill

%\begin{classcode}01\ref{(b)}3; 45B67 \textbf{(... for example; authors are requested to provide at least one 2010 Mathematics Subject Classification code)}\end{classcode}

\section{Introduction}

The multivariate multiple linear regression model for a  $q$-dimensional  response $Y=(Y_1,\ldots, Y_q)^T$
and a $p$-dimensional predictor vector $X=(X_1,\ldots, X_p)^T$ postulates that  $Y=B X+\epsilon$,
where $B$ is a $q\times p$ matrix and
$\epsilon= (\epsilon_1,\ldots, \epsilon_q)^T$  is the error term,  with $\E(\epsilon)=0$ and
$\var(\epsilon)=\Sigma$.  Based on a random sample $(X_1, Y_1), \ldots, (X_n, Y_n)$
satisfying the model,
ordinary least squares estimates the parameter matrix
$B$ by minimizing the squared error loss function $\sum_{i=1}^n || Y_i-B X_i||^2$ to obtain
\begin{equation}
\label{ls-loss}
%\sum_{i=1}^n || Y_i-B\X_i||^2\;, \quad\hbox{to obtain } \quad
\widehat{B}_{\OLS}=\left(\sum_{i=1}^n Y_i X_i^T\right) \left(\sum_{i=1}^n X_i X_i^T\right) ^{-1}.
\end{equation}

Reduced-rank regression  introduces a rank constraint on $B$, so that~\eqref{ls-loss}
is minimized subject to the constraint  $\rank(B) \le r$, where $r< \min(p,q)$. The solution is
$\widehat{B}_{\RRR}^T=\widehat{B}_{\OLS}^T U_r U_r^T$, where $U_r$ are the first $r$
singular vectors of $\widehat{Y}^T=\widehat{B}_{\OLS}X^T$ [see, e.g., Reinsel and Velu
(1998)].

Reduced-rank regression  has attracted attention as a regularization method by introducing a shrinkage penalty on $B$. Moreover, it is used as a dimensionality reduction method as it constructs latent factors in the predictor space that explain the variance of the responses.

Anderson (1951) obtained the likelihood-ratio test of the hypothesis that the rank of $B$ is a given number and derived the associated asymptotic theory under
the assumption of normality of $Y$ and non-stochastic $X$.  Under the assumption of joint normality of $Y$ and $X$, Izenman (1975) obtained the asymptotic distribution of the estimated reduced-rank regression coefficient matrix and drew connections between reduced-rank regression, principal component analysis  and correlation analysis. Specifically, principal component analysis coincides with reduced-rank regression  when $Y=X$ (Izenman (2008)). The monograph
by Reinsel and Velu (1998) contains a comprehensive survey of the theory and history of reduced rank regression, including in time series, and its many applications.

Despite the application potential of reduced-rank regression, it has received limited attention in generalised linear models. Yee and Hastie (2003) were the first to introduce reduced-rank regression to the class of multivariate generalised linear models, which covers a wide range of data types for both the response and predictor vectors, including categorical data. Multivariate or vector generalised linear models is the topic of Yee's (2015) book, which is accompanied by the  associated \texttt{R} packages \texttt{VGAM} and \texttt{VGAMdata}. Yee and Hastie (2003) proposed an alternating estimation
algorithm, which was shown to result in the maximum likelihood estimate of the parameter matrix in reduced-rank multivariate generalised linear models by Bura, Duarte and Forzani (2016). Asymptotic theory for the restricted rank maximum likelihood estimates of the parameter matrix in multivariate GLMs has not been developed yet.

In general, a maximum likelihood estimator  is a concave M-estimator in the sense that it  maximizes the empirical mean of  a concave criterion  function.
Asymptotic theory for M-estimators defined through a concave function has received much attention.  Huber (1967), Haberman  (1989)  and Niemiro (1992) are among the classical references.  More recently, Hjort and Pollard  (2011) presented a unified  framework for the statistical theory of M-estimation for convex criterion functions that are minimized over open convex sets of a Euclidean space.
Geyer (1994) studied M-estimators restricted to a closed subset of a Euclidean space.

The rank restriction in reduced rank regression imposes constraints that have not been studied before in M-estimation as they result in neither convex nor closed parameter spaces.
In this paper we (a) develop M-estimation theory for concave criterion functions, which are maximized over parameters spaces that are neither convex nor closed, and (b) apply the results from (a) to obtain asymptotic theory for reduced rank regression estimators in generalised linear models.
Specifically, we derive the asymptotic  distribution and properties of maximum likelihood estimators in  reduced-rank multivariate generalised linear models where both the response and predictor vectors have a joint distribution. The asymptotic theory we develop covers reduced-rank regression for linear models as a special case. We show the improvement in inference the asymptotic theory offers via analysing the data set Yee and Hastie (2003) analysed.  

Throughout, for a function $f : \real^q \to \real$, $\nabla f(x)$ denotes the
row vector $\nabla f(x) =(\partial f(x) /\partial x_1, \ldots, \partial f(x) /\partial x_q)$,
$\dot f(x)$ stands for the column vector of derivatives, while $\nabla^2 f(x)$ denotes the symmetric matrix of
second order derivatives. For a vector valued function $f : \real^{q_1} \to \real^{q_2}$, $\nabla f$ denotes the $q_2 \times q_1$ matrix,
%\[
%\nabla_{\f}\f= \begin{pmatrix}
%\frac{\partial f_1}{\partial \phi_1} & \frac{\partial f_1}{\partial \phi_1} & \cdots & \frac{\partial f_{q_{1}}}{\partial \phi_q} \\
%\frac{\partial f_1}{\partial \phi_2} & \frac{\partial f_2}{\partial \phi_2} & \cdots & \frac{\partial f_{q_{1}}}{\partial \phi_q}\\
%\vdots & \vdots & \vdots & \vdots \\
%\frac{\partial f_1}{\partial \phi_{q_2}} & \frac{\partial f_2}{\partial \phi_{q_2}} & \cdots & \frac{\partial f_{q_{1}}}{\partial \phi_{q_2}}
%\end{pmatrix}
%\]
\[
\nabla f(x)= \begin{pmatrix}
\frac{\partial f_1}{\partial x_1} & \frac{\partial f_1}{\partial x_2} & \cdots & \frac{\partial f_1}{\partial x_{q_1}} \\
\frac{\partial f_2}{\partial x_1} & \frac{\partial f_2}{\partial x_2} & \cdots & \frac{\partial f_2}{\partial x_{q_1}}\\
\vdots & \vdots & \vdots & \vdots \\
\frac{\partial f_{q_2}}{\partial x_1} & \frac{\partial f_{q_2}}{\partial x_2} & \cdots & \frac{\partial f_{q_{2}}}{\partial x_{q_1}}
\end{pmatrix}
\]

\section{M-estimators }\label{M-estimators}

Let  $Z$ be a random  vector taking values in a measurable space $\mathcal Z$ and
distributed according to the law $P$. We are interested in estimating  a finite dimensional  parameter $\xi_0=\xi_0(P)$ using $n$ independent and identically
distributed copies $Z_1,Z_2, \ldots, Z_n$ of $Z$. In the sequel, we use  $Pf$ to denote the mean  of $f(Z)$; i.e.,  $Pf=\E[f(Z)]$.

Let $\Xi$ be  a subset of an Euclidean space
and $m_\Xi: Z \to \real$ be a known function. Assume that  the parameter of interest $\xi_0$ is the maximizer of the map $\xi \mapsto P m_{\xis}$ defined on  $\Xi$.
%$\real^k$.
One can estimate  $\xi_0$ by maximizing an empirical version of the optimization criterion. Specifically, given  a known function  $m_{\xis}:\mathcal{Z} \mapsto \mathbb{R}$,  define
\begin{equation}
\label{MyMn}
 \MM(\xi):=Pm_{\xis}\; \quad\hbox{and}\quad \MMn(\xi):=P_n m_{\xis}\;,
\end{equation}
where, here and throughout, $P_n$ denotes the empirical mean operator $P_n V=n^{-1}\sum_{1=1}^nV_i$. Hereafter,  $\MMn(\xi)$ and $P_n m_{\xis}$ will be used interchangeably as the criterion function, depending on which of the two is appropriate in a given setting.

Assume that $\xi_0$ is the unique maximizer of the deterministic function  $M$ defined in~\eqref{MyMn}. An \emph{M-estimator} for the criterion function $\MMn$ over $\Xi$ is defined  as
\begin{equation}
\label{M-estimator}
\widehat{\xi}_n=\widehat{\xi}_n(Z_1,\ldots,Z_n)\quad\hbox{maximizing}\quad
\MMn(\xi)=\frac{1}{n} \sum_{i=1}^nm_{\xis}(Z_i) \quad  \mbox{over $\Xi$ }.
 \end{equation}
If  the maximum of the criterion function $\MMn$ over $\Xi$ is  not attained but the supremum of $M_n$
over $\Xi$ is finite, any  value $\widehat{\xi}_n$ that almost maximizes the criterion function, in the sense that it satisfies
\begin{equation}\label{def max}
\MMn(\widehat{\xi}_n) \;\ge \;\sup_{\xis \in \Xi}
\MMn (\xi )-A_n
\end{equation}
for $A_n$ small, can be used instead.

\begin{definition}\label{ws-M}
An estimator $\widehat{\xi}_n $ that satisfies~(\ref{def max}) with $A_n=o_p(1)$   is called a weak M-estimator
 for the criterion function $\MMn$ over $\Xi$. When $A_n= o_p (n^{-1})$, $\widehat{\xi}_n$ is called a strong
M-estimator.
\end{definition}

Proposition \ref{consistency-new} in the Appendix lists the conditions for the existence, uniqueness and strong consistency of an M-estimator, as defined in~(\ref{M-estimator}),
 when $\MMn$ is concave and the parameter space is convex. Under regularity conditions, as those stated in Theorem 5.23 in van der Vaart (2000), the asymptotic expansion and distribution of a consistent strong  M-estimator [see Definition \ref{ws-M}] for $\xi_0$ is given by
 \begin{align}\label{infl1}
\sqrt{n} (\widehat{\xi}_n - \xi_0) &=
\frac{1}{\sqrt{n}} \sum_{i=1}^n IF_{\xis_0}(Z_i)+ o_p(1),
\end{align}
where $IF_{\xis_0}(Z_i)=-V_{\xis_0}^{-1} \dot{m}_{\xis_0}(Z_i) $, and $V_{\xis_0}$ is the   nonsingular symmetric second derivative matrix of $M(\xi)$ at $\xi_0$ .

\subsection{Restricted M-estimators}\label{restricted-M-estimators}

We now consider the optimization of  $\MMn$ over $\Xi^{\res}\subset \Xi$,  where  $\Xi^{\res}$ is the image of a function that is not necessarily injective. Specifically, we restrict the optimization problem to the set $\Xi^{\res}$ by requiring:
\begin{condition}
\label {C1}
%\item[\ref {C1}]
There exists an open set  $\Theta \subset \real^q $ and a map $g: \Theta \rightarrow \Xi$  such that $\xi_0 \in g (\Theta)=\Xi^{\res}$.
\end{condition}
Even when an M-estimator  for the unrestricted problem as defined in~(\ref{M-estimator}) exists, there is no a priori guarantee that the supremum is attained
when considering  the restricted problem.  Nevertheless, Lemma~\ref{existence-restricted}  establishes the existence of a restricted strong M-estimator, regardless of whether the original M-estimator is a real maximiser, weak or strong. All proofs are provided in the Appendix.

\begin{lemma}\label{existence-restricted}
Assume  there exists a weak/strong M-estimator $\widehat\xi_n$ for the criterion function $\MMn$  over $\Xi$. If Condition \ref {C1} holds,
then  there exists  a strong M-estimator $\xires$ for the criterion function $\MMn$ over $\Xi^{\res}$.
\end{lemma}

Proposition~\ref{restricted-consistency} next establishes the existence and consistency of a weak restricted  M-estimator sequence when $m_\xi(z)$ is a concave function in $\xi$
under the same assumptions as those of the unrestricted problem, stated in Proposition \ref{consistency-new} in the Appendix.

\begin{proposition}\label{restricted-consistency}
Assume that Condition~\ref {C1}  holds. Then, under the  assumptions of Proposition \ref{consistency-new} in the Appendix,
%{\color{black}{pondria under the assumptions of Proposition \ref{consitency-new}} sacaria lo de convexity assumption}
there exists a strong M-estimator of the restricted problem over $\Xi^{\res}$.
Morevoer, any strong M-estimator of the restricted problem converges to $\xi_0$ in probability.
\end{proposition}

We derive next the asymptotic distribution  of $\xires=g(\widehat \theta_n)$, with $ \widehat \theta_n \in \Theta$. The constrained estimator $\xires$ is well
defined under Lemma~\ref{existence-restricted}, even when $\widehat \theta_n$ is not uniquely determined. If $\widehat \theta_n$  were unique and $\sqrt{n}(\widehat \theta_n - \theta)$ had an asymptotic distribution, one could use a Taylor series expansion for $g$ to derive asymptotic results for $\xires$. Building on this idea, Condition \ref {C2} introduces a parametrization of a neighborhood of $\xi_0$ that allows applying standard tools in order to  obtain the asymptotic distribution of the restricted M-estimator.

\begin{condition} \label{C2}
%\item[(C.2)]
Given $\xi_0\in g(\Theta)$, there exist an open set $\mathcal M$ in $\Xi^{\res}$ with $\xi_0\in \mathcal M\subset g(\Theta)$, and $(\mathcal S, h)$, where $\mathcal S$ is an open set in $\real^{q_s}$, $q_s \leq q$,  and $h :\mathcal S \to \mathcal M$ is one-to-one, bi-continuous and twice continuously differentiable, with $\xi_0=h(s_0)$ for some $s_0\in \mathcal S$.
\end{condition}
Under the setting in Condition \ref {C2}, we will prove that $\widehat{s}_n=h^{-1}(\xires)$ is a  strong M-estimator for the
criterion function $P_nm_{h(s)}$ over $\mathcal S$. Then, we can apply  Theorem 5.23  of van der Vaart (2000) to obtain the asymptotic behavior of $\widehat{s}_n$, which, combined with a Taylor expansion of $h$ about $s_0$, yield a linear expansion for $\xires$. Finally, requiring Condition~\ref {C3}, which  relates the parametrizations $(\mathcal S,h)$ and $(\Theta, g)$, suffices to derive  the asymptotic distribution of $\xires$ in terms of $g$.

\begin{condition}\label{C3}
 %\item[(C.3)]
 Consider $(\Theta, g)$ as in Condition \ref {C1} and  $(\mathcal S, h)$ as in Condition \ref {C2}. For  each $\theta_0\in g^{-1}(\xi_0)$, $\spn \nabla g(\theta_0)=\spn \nabla h(s_0)$.
%Consider $(\mathcal S, h)$ as in \ref {C2} and let $\Tan$ denotes
% \textcolor{black}{ bla bla bla the linear space generated by }
% $\Tan=\spn \nablah(\s_0)$. For  each $\theta_0\in
% \g^{-1}(\xi_0)$, $\spn \nabla \g(\theta_0)=\Tan$.
\end{condition}
Condition \ref {C3} ensures that  $\Tan=\spn \nabla g(\theta_0)$ is well defined regardless of the fact that $g^{-1}(\xi_0)$ may contain multiple $\theta_0$'s. Moreover,
$\Tan$ also agrees with $\spn \nabla h(s_0)$. Consequently, the orthogonal projection $\Pi_{\xis_0(\Sigma)}$ onto $\Tan$ with respect to the  inner product defined by a definite positive matrix $\Sigma$ satisfies
\begin{eqnarray}\label{Pi General}
\Pi_{\xis_0(\Sigma)} &=& \nabla g(\theta_0)(\nabla g(\theta_0)^T\Sigma\ \nabla g(\theta_0))^{\dag}\nabla g(\theta_0)^T\Sigma\\
\nonumber  &=& \nabla h(s_0)( \nabla h(s_0)^T \Sigma  \nabla h(s_0))^{-1} \nabla h(s_0)^T\Sigma,
\end{eqnarray}
where $A^{\dag}$ denotes a generalised inverse of the matrix $A$.
The gradient of $g$ is not necessarily of full rank, in contrast to  the gradient of $h$.

\begin{proposition}\label{first}
Assume  that Conditions  \ref {C1},  \ref {C2} and  \ref {C3} hold.
Assume also that the unrestricted problem satisfies the regularity Conditions \ref{(a)}, \ref{(b)} and \ref{(c)} in the Appendix. %Theorem 5.23 of van de Vaart (2000), so that the linear expansion (\ref{infl1}) holds.
Then, any strong M-estimator $\widehat{\xi}_n^{\res}$ of the restricted problem that converges in probability to $\xi_0$ satisfies
\begin{align}\label{infl2}
\sqrt{n} (\widehat{\xi}_n^{\res} - \xi_0) &=
\frac{1}{\sqrt{n}} \sum_{i=1}^n \Pi_{\xis_0(-V_{\xi_0})} IF_{\xis_0}(Z_i)+ o_p(1),
\end{align}
{where $IF_{\xis_0}(Z_i)=-V_{\xis_0}^{-1} \dot{m}_{\xis_0}(Z_i) $  is the influence function of the unrestricted estimator defined in~(\ref{infl1}),   $V_{\xis_0}$ is the   nonsingular symmetric second derivative matrix of $M(\xi)$ at $\xi_0$, and  $\Pi_{\xis_0(-V_{\xi_0})}$ is defined according to  (\ref{Pi General}).}

Moreover, $\sqrt{n} (\xires - \xi_0)$ is asymptotically normal with mean zero and asymptotic variance
\begin{equation}
\label{normal asintotica}
\hbox{avar}\{\sqrt{n} (\xires - \xi_0)\}=
\Pi_{\xis_0(-V_{\xi_0})} V_{\xis_0}^{-1} P\left\{{\dot m}_{\xis_0} {\dot
m}_{\xis_0}^{T}\right\} V_{\xis_0}^{-1}\Pi_{\xis_0(-V_{\xi_0})}^T.
\end{equation}
\end{proposition}

As a side remark, we conjecture that for estimators that are maximizers of a criterion function under restrictions that satisfy Conditions  \ref {C1}, \ref {C2} and
\ref {C3}, when~(\ref{infl1}) is true, (\ref{infl2}) also holds. This can be important since the asymptotic distribution of restricted estimators will be derived directly from the asymptotic distribution of the unrestricted one.

\section{Asymptotic theory for the maximum likelihood estimator in reduced rank multivariate generalised linear models}\label{RRGLM}

In this section we show that maximum likelihood estimators in reduced-rank multivariate generalised linear models are restricted  strong M-estimators for the conditional log-likelihood. Using results in Section \ref{restricted-M-estimators}, we obtain the existence, consistency and asymptotic distribution of maximum likelihood estimators  in reduced-rank multivariate generalised linear models.

\subsection{Exponential Family}\label{exponential-family}
Let $Y=(Y_1,\ldots, Y_q)^T$ be a $q$-dimensional random vector and assume that its distribution belongs to a $k$- parameter canonical exponential family with pdf (pms)
\begin{eqnarray}\label{expfamily1}
  f_{\etas}(y) &=&  \exp\{\eta^T T(y) -\psi(\eta)\}h(y),
\end{eqnarray}
where $T(y)=(T_1(y),\ldots,T_k(y))^T $ is a vector of known real-valued functions, $h(y)\geq 0$ is a nonnegative known function and $\eta\in \mathbb{R}^{k}$ is the vector of natural parameters, taking values in
 \begin{equation}
 \label{natural-space}
 H=\{\eta\in \real^k: \int \exp\{\eta^T T(y)\}h(y)dy<\infty\},
 \end{equation}
where the integral is replaced by a sum when $Y$ is discrete.
The set $H$ of the natural parameter space is assumed to be  open and convex in $\real^k$, and $\psi$ a strictly convex function defined on $ H$.  Moreover, we assume $\psi(\eta)$ is convex and infinitely differentiable in $ H$. In particular, %(HERE REFERENCE: SABRI)
\begin{equation}\label{unito}
\nabla \psi (\eta)=\E_{\eta}^T(T(Y)) \quad
 \mbox{and}\quad
\nabla^2 \psi (\eta)=\var_{\eta} (T(Y)),\; \; \mbox{
for every $\eta\in  H$},
\end{equation}
where $\nabla^2 \psi $ is the $k \times k$ matrix of second derivatives of $\psi$.
Since $\psi$ is strictly convex, $\var_{\eta} (T(Y))$ is non-singular for every $\eta\in H$.

\subsection{Multivariate generalised linear models}\label{section1}
Let $Z=(X, Y)$ be a random vector, where now $Y \in \mathbb{R}^q$  is a multivariate response  and $X \in \mathbb{R}^p$ is a vector of predictors. The multivariate  generalised linear model postulates  that the conditional distribution of $Y$ given $X$
{belongs to some fixed exponential family and  hypothesizes that  the $k$-vector of natural parameters, which we henceforth call $\eta_x$ to emphasise the dependence on $x$,  depends linearly on  the vector of predictors.} Thus, the pdf (pms) of $Y\mid X=x$ is given by
\begin{equation}
\label{expfamily}
  f_{Y\mid X=x }(y) \;=\;  \exp\{\eta_{x}^T T(y)-\psi(\eta_{x})\}h(y),
\end{equation}
where $\eta_x \in \mathbb{R}^{k}$ depends linearly on $x$.

Frequently, a subset of the natural parameters depends on
$x$ while its complement does not.  The normal linear model with constant variance is such an example. To
accommodate this structure, we partition the vector $\eta_x$
indexing model  (\ref{expfamily}) into $\eta_{x 1}$ and $\eta_{x 2}$,
with $k_1$ and $k_2$ components, and assume that $H$, the natural parameter space of the exponential family, is $\real
^{k_1}\times H_2$, where  $ H_2$ is an open convex subset of
$\real^{k_2}$,
and assume that
%$\bold H=\bold H_1\times \bold  H_2$, where $\bold
%H_1=\real^{k_1}$  and  $\bold H_2$ is an open set in $\real^{k_2}$,
% According to this decomposition, we
% To contemplate this situation, we will subdivide the vector
% $\eta_x$
% in $\eta_{x1}$ and $\eta_{x2}$,
% with $k_1$ and $k_2$ components respectively,
%and postulate that
\begin{align}\label{modelo1}
 \eta_x&= \begin{pmatrix} \eta_{x1} \\ \eta_{x2} \end{pmatrix}=
 \begin{pmatrix} \overline{\eta}_{1}+ \beta x\\
\overline{\eta}_{2}
\end{pmatrix}
\end{align}
where $\beta \in \mathbb{R}^{k_1 \times p}$, $\overline{\eta}_{1} \in \mathbb{R}^{k_1}$ and $\overline{\eta}_{2} \in  H_2$. %,  an open subset of $\real^{k_2}$.
%and $\overline{\eta}^{T}=(\overline{\eta}_{1}^{T},\overline{\eta}_{2}^{T})$ is the constant part of $\eta_{x1}$ and $\eta_{x2}$.
%We denote the true parameter value with
%\begin{align}\label{modelo1true}
% \eta_{0x}&= \begin{pmatrix} \eta_{0x1} \\ \eta_{0x2} \end{pmatrix}=
% \begin{pmatrix} \overline{\eta}_{01}+ \beta_0 x\\
%\overline{\eta}_{02}
%\end{pmatrix}
%\end{align}
%and let $\overline{\eta}_0^{T}=(\overline{\eta}_{01}^{T},\overline{\eta}_{02}^{T})$ denote the constant part of $\eta_{0x1}$ and $\eta_{0x2}$.
%We are interested in inference about $\beta_0$ and $\overline{\eta}_0$ based on
%Let $\xi=\left(\overline{\eta}_1^T,\vecc^T(\beta),\overline{\eta}_2^T \right)^{T}$
Let  $\xi=\left(\overline{\eta}_1^T,\vecc^T(\beta),\overline{\eta}_2^T \right)^{T} \in\Xi= \real^{k_1 }\times \real^{k_1   p} \times H_2$,  denotes a generic
vector and $\xi_0$ the true parameter.
% Note that $\Xi \subseteq \real ^K$, with $K=k_1(1+r)+k_2$.
Suppose $n$ independent and identically distributed copies of $Z=(X, Y)$ satisfying (\ref{expfamily}) and (\ref{modelo1}),
with true parameter vector $\xi_0$,  are available.
Given a realisation $z_i=(x_i,y_i)$, $i = 1,\ldots, n$, %of the random zample $\Z_1,\ldots, \Z_n$,
the conditional  log-likelihood, up to a factor that does not depend on the parameter of interest, is
\begin{equation}\label{log-lik}
\mathcal{L}_n(\beta; \overline{\eta} )=%\frac{1}{n}\sum_{i=1}^n \left\{(\overline{\eta}_1^{T}+(\beta x_i)^{T}, \overline{\eta}_2^{T})T(y_i) -\psi((\overline{\eta}_1^{T}+(\beta x_i)^{T}, \overline{\eta}_2^{T}))\right\}.
 \frac{1}{n}\sum_{i=1}^n \left\{\eta_{x_i}^{T}T(y_i)
 -\psi(\eta_{x_i})\right\}.
\end{equation}
Let
\begin{equation}\label{m de glm}
m_{(\betas; \overline{\etas})}(z) =%\left(\overline{\eta}_1^{T}+(\beta x)^{T}, \overline{\eta}_2^{T}\right)T(y) -\psi\left((\overline{\eta}_1^{T}+(\beta x)^{T}, \overline{\eta}_2^{T})\right).
 \eta_x^{T} T(y)-\psi (\eta_x ).
\end{equation}
By definition~(\ref{M-estimator}), the maximum likelihood estimator (MLE), if exits, of the parameter indexing model (\ref{expfamily}) subject to (\ref{modelo1}) is an M-estimator.

Theorem~\ref{teorema GLM} next establishes the existence,
consistency and asymptotic normality of  $\widehat\xi_n$,  the MLE  of
 $\xi_0$.

\begin{theorem}\label{teorema GLM}
Assume that  $Z=(X, Y)$ satisfies model  (\ref{expfamily}) subject to (\ref{modelo1})
with true parameter $\xi_0$.  Under regularity conditions \textcolor{black}{(\ref{concentrado}), (\ref{forprop1}), (\ref{nabla-psi-bounded}) and (\ref{hessiano})} in the Appendix,
the maximum likelihood estimate of $\xi_0$, $\widehat\xi_n$,  exists, is unique and converges in probability to $\xi_0$.
Moreover, $\sqrt{n}\;  (\widehat\xi_n-\xi_0)$  is asymptotically normal with covariance matrix
%\begin{equation}\label{W}
%W_{\xi_0}=\left[ \E\left\{F(X)^{T} %\nabla^{2}\psi(\eta_{x_o})F(X)\right\}\right]^{-1},
%\end{equation}
\begin{equation}\label{W}
W_{\xi_0}=\left[ \E\left\{F(X)^{T} \nabla^{2}\psi(F(X)\xi_0)F(X)\right\}\right]^{-1},
\end{equation}
where \[F(x)=\left(
            \begin{array}{cc}
            \left({{(1,x^T)}} \otimes I_{k_1}\right) & 0 \\
              0 & I_{k_2} \\
            \end{array}
          \right),\]
          %$\eta_{x_0}= %(\overline{\eta}_{01}^{T}+(\beta_0X)^{T},\overline{\eta}_{02})$
          and $\nabla^{2}\psi$
         {{was defined in (\ref{unito})}}

\end{theorem}

\subsection{Partial reduced  rank multivariate  generalised linear models}

When the number of natural parameters or the number of predictors is large, the precision of the estimation and/or the interpretation of results can be adversely affected. A way to address this is to assume that the parameters live in a lower dimensional space. That is,  we assume that the vector of predictors can be partitioned as
$x=(x_1^T,x_2^T)^T$ with $x_1 \in \mathbb{R}^{r}$ and $x_2 \in
\mathbb{R}^{p-r}$, and  that the parameter corresponding to $x_1 $, $\beta_{1} \in \real^{k_1 \times r}$ has rank $d< \min \{k_1,r\}$.
%
%
%in such a way that the reduced rank model
%involves only the first component of the predictors; that is, a subset of the regressors have a reduced-rank representation.
In this way,
the natural parameters $\eta_x$ in  (\ref{expfamily})
are related to the predictors via
\begin{align}\label{modelo2}
 \eta_x&= \begin{pmatrix} \eta_{x1} \\ \eta_{x2} \end{pmatrix}=
 \begin{pmatrix} \overline{\eta}_{1}+ \beta_{1} x_1+ \beta_{2} x_2 \\
\overline{\eta}_{2}
\end{pmatrix}
\end{align}
where  $\beta_{1} \in \mathbb{R}^{k_1 \times r}_d$, the set of
matrices in $\real^{k_1\times p}$ of rank $d \le \min \{k_1,{ {r}}\}$, while $\beta_{2} \in \mathbb{R}^{k_1
\times (p-r)}$, and $\beta=(\beta_{1},\beta_{2})$.

Following Yee and Hastie (2003), we refer to the exponential conditional model
(\ref{expfamily}) subject to
the restrictions imposed in (\ref{modelo2}) as \textit{partial reduced rank multivariate generalised linear  model}. The reduced-rank multivariate generalised linear model is the special case with $\beta_2=0$ in the partial reduced rank multivariate generalised linear  model.

To obtain the asymptotic distribution of the M-estimators for this reduced model, we will show that Conditions \ref {C1}, \ref {C2} and \ref {C3} are satisfied for $\Xi^{\res}$, $(\Theta, g)$, $\mathcal M$ and $(\mathcal S, h)$, which are defined next. % to account for model (\ref{modelo2}).
To maintain consistency with notation introduced in Section \ref{restricted-M-estimators}, we vectorise each  matrix involved in the parametrisation of
our model  and reformulate accordingly the  parameter space for  each vectorised  object. With this understanding, we use the symbol $\cong$ to
indicate that a matrix space component in a  product space is identified with its image through the operator $\vecc:\real^{m\times n}\to \real^{mn}$.
 In the sequel, to keep the notation as simple as possible, we concatenate column vectors
without transposing them; i.e.,
we write $(a,b)$ for $(a^T, b^T)^T$. Moreover, we write $\xi=\left(\overline{\eta}_1,\beta,\overline{\eta}_2 \right)$, with the understanding that $\beta$  stands for $\vecc(\beta)$.

For the non-restricted problem, $\beta_{1}$ belongs to $\real^{k_1\times r}$, so that the entire parameter
$\xi=\left(\overline{\eta}_1 , \beta_1 , \beta_2 ,\overline{\eta}_2  \right)$  belongs to
\begin{equation}
 \label{Xi-example}
\Xi\cong\real^{k_1}\times \real^{k_1\times r}\times\real^{k_1\times(p-r)}\times H_2.
\end{equation}
However, for the restricted problem, we assume that  the true parameter $\xi_0=\left(\overline{\eta}_{01} , \beta_{01} , \beta_{02} ,\overline{\eta}_{02}  \right)$ belongs to
\begin{equation}\label{Xires-example}
\Xi^{\res}\cong\real^{k_1}\times \real_d^{k_1\times r}\times\real^{k_1\times(p- r)}\times H_2.
\end{equation}
Let
\begin{equation}\label{Theta-example}
\Theta \cong \real^{k_1}\times \left\{\real_d^{k_1\times d}\times \real_d^{d\times r}\right\}\times\real^{k_1\times(p- r)}\times
H_2
\end{equation}
and consider $g:\Theta \to \Xi^{\res}$, with $\left(\overline{\eta}_1,(S,T),\beta_2,\overline{\eta}_2\right)\mapsto
\left(\overline{\eta}_1,S T,\beta_2,\overline{\eta}_2\right)$.
Without loss of generality, we  assume  that $\beta_{01} \in \real^{k_1\times r}_{\firsts,d}$, the set of matrices in $\real^{k_1\times r}_d$ whose  first $d$ rows are linearly
independent. Therefore,
\begin{eqnarray}\label{beta 0}
\beta_{01} = \left(
\begin{array}{c} I_d
\\ A_0 \end{array} \right) B_0, & &\textnormal{ with } A_0 \in \real^{(k_1-d) \times d} \textnormal{  and  }B_0 \in \real^{d\times r}_d.
\end{eqnarray}
{This is trivial since $\beta_{01}\in \mathbb R_d^{k_1\times r}$ and its  first $d$ rows linearly independent. }
Consider
\begin{equation}\label{M-example}
\mathcal M\cong \real^{k_1}\times \real_{\firsts,d}^{k_1\times r}\times\real^{k_1\times(p- r)}\times
H_2, \quad\hbox{and so}\quad \xi_0\in \mathcal M\subset \Xi^{\res}.
\end{equation}
Finally, let
\begin{equation}
 \label{S-example}
 \mathcal S\cong \real^{k_1}\times \left\{ \real^{(k_1-d) \times d} \times \real^{d\times r}_d\right\}\times\real^{k_1\times(p- r)}\times H_2,
 \end{equation}
and $h: \mathcal S \to  \mathcal M$ be the map
\begin{equation}\label{pfcn}
\left(\overline{\eta}_1,(A,B),\beta_2,\overline{\eta}_2\right)\quad\mapsto \quad \left(\overline{\eta}_1, \left(
\begin{array}{c} B
\\A B \end{array} \right),\beta_2,\overline{\eta}_2\right).
\end{equation}

\begin{proposition}\label{condiciones de RR}
Conditions \ref {C1}, \ref {C2} and \ref {C3} are satisfied for $\xi_0 $, $\Xi$, $\Xi^{\res}$, $(\Theta, g)$ and $(\mathcal{S}, h)$ defined in~(\ref{Xi-example})--(\ref{pfcn})%, (\ref{Xires-example}), (\ref{Theta-example}), (\ref{beta 0}), (\ref{M-example}), (\ref{S-example}) and (\ref{pfcn})
, respectively.
\end{proposition}

Under the rank restriction on $\beta_{1}$ in (\ref{modelo2}), the existence of the maximum likelihood estimate  cannot be guaranteed in the sense of an M-estimator as defined in (\ref{M-estimator}) with $m_{\xis}=m_{(\betas; \overline{\etas})}$ in (\ref{m de glm}), and $\Xi$ replaced by $\Xi^{\res}$ in (\ref{Xi-example}). However, using Lemma~\ref{existence-restricted} we can work with a strong M-estimator sequence for the criterion function $P_nm_{(\betas; \overline{\etas})}$ over $\Xi^{\res}$.
Theorem~\ref{teorema GLM RR} states our main contribution.

\begin{theorem}\label{teorema GLM RR}
Let $\xi_0=\left(\overline{\eta}_{01}, \beta_{01}, \beta_{02},\overline{\eta}_{02} \right) $ denote the true parameter value of $\xi=\left(\overline{\eta}_1, \beta_1 , \beta_2,\overline{\eta}_2 \right) $. Assume that  $Z=(X, Y)$ satisfies model  (\ref{expfamily}), subject to (\ref{modelo2}) with $\xi_0\in \Xi^{\res}$ defined in (\ref{Xires-example}). Then, there exists a strong maximising sequence for the criterion function $P_nm_{\xis}$ over $\Xi^{\res}$ \textcolor{black}{for  $m_{\xis}=m_{(\betas; \overline{\etas})}$ defined in (\ref{m de glm}).}
% for $m_{\xis}=m_{(\betas; \overline{\etas})}$
%defined in (\ref{m de glm-rr}).
Moreover, any weak M-estimator  sequence $\{\xires\}$  converges to $\xi_0$ in probability.

If  $\{\xires\}$ is a strong M-estimator  sequence, then $\sqrt{n}(\xires-\xi_0)$ is asymptotically normal with covariance matrix
\begin{eqnarray*}
  \hbox{avar}\{\sqrt{n} (\xires - \xi_0)\}
 &=&\Pi_{\xis_0(W_{\xis_0}^{-1})}W_{\xis_0}
\end{eqnarray*}
where $W_{\xis_0}$ is defined in (\ref{W})

\end{theorem}

\begin{remark}
Since $\Pi_{\xis_0(W_{\xis_0}^{-1})}$ is a projection,
$\Pi_{\xis_0(W_{\xis_0}^{-1})}W_{\xis_0} \le W_{\xis_0} $.  That is, the eigenvalues of $W_{\xis_0} -\Pi_{\xis_0(W_{\xis_0}^{-1})}W_{\xis_0} $ are non-negative, so that using partial reduced-rank multivariate generalised linear models results in efficiency gain.
%For symmetric matrices $A$, $B$, $A\le B $ means that the all the eigenvalues of $B-A $ are positive definite
\end{remark}

\section{Application: Marital status in a workforce study}\label{application}
Yee and Hastie (2003)  analyse data from a self-administered questionnaire
collected in a large New Zealand workforce observational study conducted during
1992-1993. For homogeneity, the analysis was restricted to a subset of 4105 European
males with no missing values in any of the variables used.
Yee and Hastie were interested in
exploring whether certain lifestyle and psychological variables were associated with
marital status, especially separation/divorce. The response variable is $Y=$ marital status,
with levels $1 = $ single, $2  = $ separated or divorced, $3 =$ widower, and
$4 =$ married or living with a partner. The married/partnered are the reference
group. Data on  14 regressors were collected, 12 of which  are binary ($1/0$ for presence/absence, respectively). These have
been coded so that their presence is negative healthwise. Their goal was  to
investigate if and how these $12$ {\it unhealthy} variables  were related to $Y$, adjusting for age
and level of education. The variables are described in Table~\ref{vardescr}. 

\begin{table}[ht!]
\centering
\begin{tabular}{lp{11cm}}
  \toprule[1.5pt]
  \head{Variable Name} & \head{Description} \\\midrule
 \texttt{marital} &  Marital status. 1=single, 2=separated or divorced, 3=widower, and 4=married or
living with a partner\\
 \texttt{age30} & Age - 30, in years\\
 \texttt{logedu1} & $\log(1+$ years of education at secondary or high school)\\
 \texttt{binge} &  In the last 3 months what is the largest number of drinks that you had on any one day?
(1=20 or more, 0=less than 20)\\
 \texttt{smokenow} & Current smoker?\\
\texttt{sun} &  Does not usually wear a hat, shirt or suntan lotion when outside during summer\\
\texttt{nerves} &  Do you suffer from `nerves'?\\
\texttt{nervous} &  Would you call yourself a `nervous' person?\\
\texttt{hurt} &  Are your feelings easily hurt?\\
\texttt{tense} &  Would you call yourself tense or `highly strung'? \\
\texttt{miserable} &  Do you feel `just miserable' for no reason?\\
\texttt{fedup} &  Do you often feel `fed-up'?\\
\texttt{worry} &  Do you worry about awful things that might happen?\\
\texttt{worrier} &  Are you a worrier?\\
\texttt{mood} & Does your mood often go up and down?\\
  \bottomrule[1.5pt]
\end{tabular}
\caption{Variables used in the workforce study}
\label{vardescr}
\end{table}

A categorical response $Y$ taking values $1,2,3,4 $ with probabilities $\pr(Y=i)=p_i$ can be expressed as a multinomial vector to fit the generalised linear model   presented in this paper, $Y=(Y_1,Y_2,Y_3,Y_4)^T$, where $Y_i=1$ if $Y=i$ and $Y_i=0$ otherwise, and $\sum_{i=1}^{4}Y_i=1$. 
The pdf of $Y$ can be written as 
\begin{equation}
\label{multinomial}
 f_{Y}(y) \;=\;  
\exp\{\eta^T T(y)-\psi(\eta)\} h(y)
\end{equation}
where $y=(y_{1}, y_{2}, y_{3}, y_{4})$, $T(y) = (y_1, y_2, y_3)$,   $\eta=(\eta_1,\eta_2,\eta_3)\in\mathbb R^3$,  $\psi(\eta)=
1+\sum_{i=1}^3 \exp(\eta_i)$ and $h(y) = 1/({y_1! y_2!  y_3 !(1-y_1-y_2-y_3)! })$. The natural parameter  $\eta=(\eta_1,\eta_2,\eta_3)\in\mathbb R^3$, is related to the pdf of $Y$ through   the identity  $\eta_i=\log (p_i/p_4)$, $i=1,2,3$.

Let $X$ be the vector of predictor variables and consider $p_i=p_i(x)=\pr(Y_i=1 \vert X=x)$. The dependence of $p_i$ on $x$ will not be made explicit in the notation.
As in Yee and Hastie (2003) we fit a multinomial regression model 
 to $Y|X=x$, as in \eqref{multinomial}, with $\eta_x$ 

\begin{equation}
\eta_x=  \left(\begin{array}{c}
 \log(p_1/p_4) \\
 \log(p_2/p_4)\\
  \log(p_3/p_4)
\end{array}\right) = \overline{\eta} +  \beta x
\end{equation}
where $ \overline{\eta} \in\mathbb{R}^3$ is the intercept and $ \beta  \in \mathbb{R}^{3 \times 14}$ is the coefficient matrix, so that there are $3 \times 14 +3=42+3=45$   parameters to estimate.
When a multinomial linear model is  fitted to the data at level 0.05,  \texttt{age30} and \texttt{binge} are significant for $\log(p_1/p_4)$, \texttt{smokehow} and \texttt{tense} for $\log(p_2/p_4)$, and only \texttt{age30} is significant for $\log(p_3/p_4)$. 

We next fitted a partial reduced rank multivariate generalised linear model, where the two continuous variables, \texttt{age30} and \texttt{logedu1}, were not subject to restriction. That is, 
\begin{equation}\label{modelo multinomial reducido}
\eta_x   = \overline{\eta} + \beta  x = \overline{\eta} +  \beta_1 x_1  + \beta_2 x_2 ,
\end{equation}
where $x_2 $ represent the continuous variables and $x_1 $ the 12 binary predictors. The AIC criterion estimates the rank of $\beta_1$ in \eqref{modelo multinomial reducido} to be one [see  Yee and Hastie (2003)]. Using the asymptotic results from the current paper, Duarte (2016) developed a test based on Bura and Yang (2011) that also estimates the dimension to be 1. Therefore, in our notation, $q=4$, $k=k_1=3$, $p=14$, $r=2$, $d=1$, and $\beta_1 = AB$, $A : 3\times 1$ y $B: 1 \times 12$,
$\beta_2 :3 \times 2$ and $\overline{\eta}: 3 \times 1$. The rank restriction results in a drastic reduction in the total number of parameters from 45 to 24.

 The reduction in the estimation burden is also reflected in how tight  the confidence intervals are compared with those in the unrestricted model, as can be seen in Tables
 \ref{two}, \ref{three} and Table~2 in Yee and Hastie (2003).   As a consequence the variables
 \texttt{nervous}, \texttt{hurt}, which are not significant in the unrestricted generalised linear model, are significant in the reduced \eqref{modelo multinomial reducido}. Furthermore, some variables, such as \texttt{binge}, \texttt{smokenow}, \texttt{nervous} and \texttt{tense}, are now significant for all responses.  
 
 All significant coefficients are positive. These correspond to the variables \texttt{binge, smokenow, nervous, tense} and \texttt{hurt} only for widowers. Since the positive value of the binary variables indicates poor lifestyle and negative psychological characteristics, our analysis concludes that for men  with these features, the chance of being single, divorced or widowed is higher than the chance of being married, adjusting for age and education.
Also, the coefficients corresponding to the response $ \log (p_3 / p_4) $ are twice as large as those of $ \log (p_1 / p_4) $, suggesting the effect of the predictors differs in each group. All computations were performed using the \texttt{R} package \texttt{VGAM}, developed by Yee (2017).

 \begin{table}[h]
  \centering
\begin{small} 
\begin{tabular}{@{}lrrrr@{}}
 \toprule
 Variable & &$\log(p_1/p_4)$ & $\log(p_2/p_4)$ & $\log(p_3/p_4)$ \\
 \midrule
\texttt{Intercept}   \\
 & {\scriptsize GLM }           & -1.762 * &  -2.699 * &   -6.711 *\\
 &         & (0.377)   & (0.383)   &  (1.047)  \\
  & {\scriptsize PRR-GLM} & -1.573* & -2.921* & -6.123 *\\
   &       &(0.388)  & (0.396) & (1.106)\\
 \texttt{age30}   \\          
 & {\scriptsize GLM }               & -0.191 * &   0.012   &  0.086 * \\
 &         & (0.008)  &  (0.008)  & (0.023)  \\
   & {\scriptsize PRR-GLM}   & -0.190* &  0.012 & 0.077*\\
    &      &(0.008)  & (0.008) & (0.024)\\
\texttt{logedu1 }   \\       
& {\scriptsize GLM }  & 0.338  & -0.365 &  -0.089  \\
   &       & (0.225) &  (0.213) & (0.559)\\
    & {\scriptsize PRR-GLM}     & 0.254  & -0.316 &-0.198\\
  &        &(0.228)  & (0.214) & (0.566)\\
  \bottomrule
 \end{tabular}
 \end{small}
  \caption{Estimators with  standard errors for $\beta_1$  for the generalised linear model (first two rows) and the reduced generalised linear model (last two rows). $^*$ indicates significant at 5\% level.}
 \label{one}
 \end{table}

\begin{table}[ht!]
  \centering
\begin{footnotesize}
\begin{tabular}{@{}lrrrr@{}}
\toprule
 Variable & & $\log(p_1/p_4)$ & $\log(p_2/p_4)$ & $\log(p_3/p_4)$ \\
 \midrule
\texttt{binge}   \\
&  {\tiny GLM}  & 0.801*  & 0.318 &1.127\\
 &          &(0.143)  & (0.256) & (0.670)\\
 & {\tiny PRR-GLM }        & 0.569 * & 0.786 * & 1.114 *\\
 &          &(0.125)  & (0.196) & (0.409)\\
\texttt{smokenow}   \\       
& {\tiny GLM}     & 0.022 & 0.501 * & 0.654\\
&           &(0.126)  & (0.157) & (0.469)\\
& {\tiny PRR-GLM }    & 0.222 * &   0.306 * &  0.434 * \\
  &        & (0.088) & (0.119) & (0.208)\\
\texttt{sun}   \\ 
& {\tiny GLM}        & -0.066  & 0.120  & -0.088 \\
 &          &(0.122)  & (0.161) & (0.518)\\
  & {\tiny PRR-GLM }            & 0.011  & 0.015  & 0.021 \\
   &       & (0.084) & (0.116) & (0.164)\\
 \texttt{nerves}   \\
& {\tiny GLM}        & -0.102   &   0.123   &  -1.456\\
 &          &(0.138)  & (0.198) & (0.841)\\
 & {\tiny PRR-GLM }                    & -0.054   &   -0.074   &  -0.105\\
&          &  (0.101)& (0.139) &(0.197)\\
 \texttt{nervous}   \\
 & {\tiny GLM}       & 0.297 &     0.353     &   1.007 \\
&           &(0.169)  & (0.228) & (0.665)\\
 & {\tiny PRR-GLM }        & 0.312 *&     0.430 *    &   0.609 * \\
  &        & (0.124) & (0.168) & (0.291)\\
 \texttt{hurt}   \\    
 & {\tiny GLM}      & 0.184     &  0.210  &   0.483 \\
 &          &(0.126)  & (0.167) & (0.501)\\
 & {\tiny PRR-GLM }    & 0.180     &  0.248  &   0.352 * \\
   &       &(0.089)& (0.122) & (0.199)\\
\bottomrule
 \end{tabular}
 
\end{footnotesize}
 %\vspace{0.3cm}
 \caption{MLEs with their standard errors in parentheses for the full rank generalised linear model (first two  rows) and the partial reduced rank generalised linear model (last two rows). $^*$ indicates significant at 5\% level.}

 \label{two}
 \end{table}

 \begin{table}[ht!]
  \centering
\begin{footnotesize}
\begin{tabular}{@{}lrrrr@{}}
\toprule
 Variable & & $\log(p_1/p_4)$ & $\log(p_2/p_4)$ & $\log(p_3/p_4)$ \\
 \midrule
 \texttt{tense}   \\    
 & {\tiny GLM}      & 0.166 &    0.483  * &  1.108\\
 &          &(0.176)  & (0.214) & (0.612)\\
& {\tiny PRR-GLM }       & 0.302 *&    0.416 *   &  0.590 *\\
  &        & (0.122)& (0.163) &(0.284)\\
\texttt{miserable}   \\    
& {\tiny GLM}      & -0.050  &  0.128   &  -0.093\\
   &        &(0.138)  & (0.178) & (0.613)\\
& {\tiny PRR-GLM }   & 0.019  &  0.0268   &  0.038\\
 &         & (0.094)&(0.129) & (0.185)\\
\texttt{fedup }   \\
& {\tiny GLM}          & 0.112  &    0.249  &   -0.214\\
 &          &(0.122)  & (0.171) & (0.548)\\
& {\tiny PRR-GLM }    & 0.117  &    0.161  &   0.229\\
  &        & (0.094) & (0.129)& (0.185)\\
\texttt{worry }   \\
& {\tiny GLM}        & 0.113  & -0.102    & -0.548\\
  &        &(0.145)  & (0.209) & (0.818)\\
& {\tiny PRR-GLM }    & 0.003  & 0.004    & 0.005\\
  &        & (0.106) & (0.146) & (0.207)\\
\texttt{worrier }   \\   
& {\tiny GLM}    & -0.027 &-0.243  &  -0.548 \\
    &       &(0.131)  & (0.180) & (0.550)\\
& {\tiny PRR-GLM }   & -0.116 &-0.160  &  -0.227 \\
  &       & (0.092)& (0.128)& (0.193)\\
\texttt{mood}   \\          
& {\tiny GLM}         & -0.111 &   0.092   &  -0.193 \\
  &         &(0.123)  & (0.171) & (0.553)\\
 & {\tiny PRR-GLM }           & -0.037  &   -0.052   &  -0.073 \\
&         & (0.087) & (0.120) & (0.172)\\
\bottomrule
 \end{tabular}
 %\vspace{0.3cm}
 \caption{MLEs with their standard errors in parentheses for the full rank generalised linear model (first two  rows) and the partial reduced rank generalised linear model (last two rows). $^*$ indicates significant at 5\% level.}
 
\end{footnotesize}
 \label{three}
 \end{table}

\FloatBarrier
\section{Discussion}
With the exception of the work of Yee on vector generalised linear
models (VGLMs) (Yee and Wild, (1996), Yee and Hastie (2003),  Yee (2017)) reduced-rank regression has been almost exclusively restricted to data where the response variable is continuous. Estimation in reduced-rank multinomial logit models (RR-MLM) was studied in  Yee and Hastie (2003), but no distribution results for the estimators were obtained.   
In this paper we fill this gap by developing asymptotic theory for the restricted rank maximum likelihood estimates of the parameter matrix in multivariate GLMs.

To illustrate the potential impact of our results, we refer to the real data analysis example in Section~\ref{application}. In order to assess the significance of the predictors, Yee and Hastie (2003) 
calculate the standard errors for the coefficient matrix factors, $  A $ and
$ B $,  independently and can only infer about the significance of the components of the matrix $A$ and the components of the matrix $B$ separately. 
The asymptotic distribution for either estimator is obtained assuming that the other is fixed and known.
 In this way, they first analyse $ \nu =  B x_1 $ to check which predictors are significant and then $A \nu$ to examine how they influence each response.
 Their standard errors are ad-hoc and it is unclear what the product of
standard errors measures as relates to the significance of the
product of the components of the coefficient matrix $\beta_1=AB $.
Moreover, this practical ad-hoc approach cannot readily be extended when $d>1$.

Using  the results of Theorem \ref{teorema GLM RR}, we can obtain
the  errors of each component of the coefficient matrices $A$, $B $ simultaneously, and assess the statistical 
significance of each predictor on each response. Using the ad-hoc approach of  Yee and Hastie (2003), a predictor can only be found to be significant across all responses. For example,  Yee and Hastie (2003) find the predictor \texttt{hurt} to be significant for all three
groups (single, divorced/separated, widower). On the other hand, we can assess the significance of any response/predictor combination. Thus, we find \texttt{hurt} to be significant only for widower but not for single or separated/divorced men groups [see  Table \ref{two}].

\section{Proofs}\label{app}

 The consistency of M-estimators has long been established [see, for instance,  Theorem 5.7 in van der Vaart  (2000)]. The functions $\MM(\xi)$ and  $\MMn(\xi)$ are defined in (\ref{MyMn}).
Typically, the proof for the consistency  of M-estimators   assumes that $\xi_0$,  the parameter of interest,  is a \textit{well-separated} point of maximum of $M$, which is ascertained by assumptions (a) and (b) of Proposition~\ref{consistency-new}. Assumption (c) of Proposition~\ref{consistency-new} yields uniform consistency of $M_n$ as an estimator of $M$, a property needed in order  to establish the consistency of M-estimators.

\begin{proposition}\label{consistency-new}
Assume
 (a) $m_{\xis}(z)$ is a strictly concave function  in $\xi \in \Xi$, where $\Xi$ is a convex open  subset of a Euclidean space; (b) the function $M(\xi)$ is well defined for any $\xi\in \Xi$ and has a unique maximizer $\xi_0$; that is, $M(\xi_0)>M(\xi)$ for any $\xi\not=\xi_0$; and (c) for any compact set $\mathcal K$  in $\Xi$,
\begin{equation}\label{envelope}
 \E\left[ \sup_{\xi \in \mathcal K}\left\vert m _{\xis} (Z)\right\vert \right]<\infty.
\end{equation}
Then, for each $n\in {\mathbb N}$, there exists a unique M-estimator $\widehat{\xi}_n$  for the criterion function $\MMn$ over $\Xi$. Moreover, $\widehat{\xi}_n \to \xi_0$  a.e., as $n \to \infty$.
 \end{proposition}

\begin{proof} %[Proof of Proposition \ref{consistency-new}]
For each compact subset $\mathcal K $ of $\Xi$, $\{m_{\xis}: \xi\in\mathcal K\}$
is a collection of measurable functions which, by assumption (c),  has an integrable envelope. Moreover,
for each fixed $z$, the map $\xi \mapsto m_{\xis}(z)$ is continuous, since it is concave and  defined on the open set $\Xi$.
As stated in Example 19.8 of van der Vaart (2000), these conditions  guarantee that the class is Glivenko-Cantelli. That is,
\begin{equation}\label{convergencia}
\mbox{pr}\left(\lim_{n\to \infty}\sup_{\xi \in \mathcal K}\vert P_n m _{\xis} -P m_{\xis}\vert= 0\right)=1.
\end{equation}
 We need to prove that there exist a unique maximizer of $M_n( \xi) = P_n \xi,$ and that it  converges to the maximizer of $M(\xi) = P m_{\xi}$.
We first consider the deterministic case ignoring  for the moment  that $\{M_n\}$ is a sequence of   random functions.

Since  $\xi_0$ belongs to  the open set $\Xi$, there exists  $\varepsilon_0>0$ such that the closed ball  $B[\xi_0,\varepsilon_0]$ is contained in $\Xi$.
Uniform convergence of $\{M_n\}$ to $M$ over $\mathcal K=\{\xi:\Vert \xi-\xi_0\Vert =\varepsilon_0\}$
%$B[\xi_0,\varepsilon_0]$  %on a set $\Omega_1$ of total probability
guarantees that
\begin{equation}\label{infimos}
\lim_{n\to\infty} \sup_{\Vert \xis-\xis_0\Vert=\varepsilon_0} \left\{ M_n(\xi)- M_n(\xi_0)\right\} = \sup_{\Vert \xis-\xis_0\Vert = \varepsilon_0}\left\{M(\xi)-
M(\xi_0)\right\}<0 \quad %\hbox{on $\Omega_1$}
\end{equation}
because $M(\xi_0)>M(\xi)$ for any $\xi\not=\xi_0$. Then, for $n\geq n_0(\varepsilon_0)$, %large enough,
\begin{equation}\label{adentro}
\sup_{\Vert \xis-\xis_0\Vert =\varepsilon_0} M_n(\xi)- M_n(\xi_0)<0.
\end{equation}
Let $\zeta_n(\xi)=M_n(\xi)- M_n(\xi_0)$.  Since $M_n$ is concave
and continuous, $\zeta_n$ attains its maximum over the compact set $B[\xi_0,\varepsilon_0]$, which we denote by $\widehat{\xi}_n$.
{Note that} $\zeta_n(\xi_0)=0$  and $\zeta_n$ is strictly smaller than zero in the boundary of the ball, as shown in (\ref{adentro}); {therefore, } we conclude that
$\widehat{\xi}_n\in B(\xi_0,\varepsilon_0)$, so that $\widehat{\xi}_n$ is a local maximum for $\zeta_n$.

%By the concavity of $\h_n$, $\widehat{\xi}_n$  is aglobal maximum. To prove this last assertion,
Let $\xi$ satisfy $\Vert \xi-\xi_0\Vert > \varepsilon_0$. The convexity of $\Xi$ implies there exists $t\in (0,1) $ such that $\widetilde{\xi}= (1-t)\xi_0 + t \xi$  satisfies $\Vert \widetilde{\xi}-\xi_0\Vert=\varepsilon_0$, and therefore
\begin{equation}\label{maximo-total}
\zeta_n(\widehat{\xi}_n)\geq \zeta_n(\xi_0)=  0 >\zeta_n(\widetilde{\xi})\geq t\zeta_n(\xi)+(1-t)\zeta_n(\xi_0)=t\zeta_n(\xi),
\end{equation}
implying that $\zeta_n(\xi)<0\leq \zeta_n(\widehat \xi_n)$. Therefore, the maximum $\widehat{\xi}_n \in B(\xi_0,\varepsilon_0)$ is global. The  strict concavity of $M_n$ guarantees that such global maximum is unique, thus  $\widehat\xi_n$ is the unique solution to the optimization problem in (\ref{M-estimator}). By
repeating this argument for any $\varepsilon<\varepsilon_0$, we prove the  convergence of the sequence $\{\widehat\xi_n\}$ to $\xi_0$.% a.e.

Turning to the  stochastic case, the uniform convergence of $M_n$ to $M$ over  $\mathcal K= B[\xi_0,\varepsilon_0]$  on a set $\Omega_1$, with $\pr (\Omega_1)=1$, as assumed in~(\ref{convergencia}), guarantees  the deterministic result can be  applied to any element of $\Omega_1$,
% $$\Omega_1=\bigcup_{n_0}\bigcap_{n\geq n_0} \left\{ \sup_{\Vert \xis-\xis_0\Vert
% =\varepsilon_0} \left\{ M_n(\xi)- M_n(\xi_0)\right\} <0\right\},$$
%with $P(\Omega_1)=1$,
which obtains the result.
\end{proof}

%\medskip

\begin{proof}[Proof of Lemma \ref{existence-restricted}]
Let   $\{\widehat  \xi_n\}$ be any   (weak/strong)  M-estimator  sequence  of the  unconstrained maximization problem.
Since  $\MMn(\widehat{\xi}_n) \ge \sup_{\xis \in \Xi} \MMn(\xi) - A_n$
with $A_n \rightarrow 0$, we have
% $\sup_{\xis \in \Xi} \MMn(\xi) \leq \MMn(\widehat{\xi}_n) +A_n$. Therefore,
\begin{equation}
1=\lim_{n \rightarrow \infty} \mbox{pr}\left(\sup_{\xis \in \Xi} P_n m_{\xis} <\infty\right) \leq \lim_{n \rightarrow \infty}  \mbox{pr}\left(\sup_{\xis\in \Xi^{\res} }P_n m_{\xis} <\infty\right).
\end{equation}
Define
\begin{equation}\label{existencerestricted}
 \Omega_n:=\left\{\sup_{\xis\in \Xi^{\res} }P_n m_{\xis}<\infty \right\}.
\end{equation}
For all $n$,  there exists $\xires$  such that
\begin{equation}
\MMn(\xires) \geq \sup_{\xis\in \Xi^{\res}} \MMn(\xi) -
\frac{1}{n^{2}}
\end{equation}
on $\Omega_n$. Let $\xires\doteq 0$ on $\Omega_n^{c}$. Then, since
$\mbox{pr}(\Omega_n)\rightarrow 1$, $\{\xires\}$ is a strong M-estimator for the criterion function $\MMn$ over $\Xi^{\res}$.
\end{proof}

\begin{proof}[Proof of Proposition \ref{restricted-consistency}]
In Proposition \ref{consistency-new} we have  established the existence of a unique maximizer  $\widehat\xi_n$ for the criterion function $M_n$ over $\Xi$.
We can now invoke Lemma \ref{existence-restricted} to guarantee
the existence of $\xires$, a strong M-estimator for the criterion function $M_n$ over $\Xi^{res}$. Let $\{\xires\}$ be any  strong M-estimator for the criterion function $\MMn$ over $\Xi^{\res}$. We start from the deterministic case:
\begin{equation}
 \label{deter}
 \MMn(\xires) \ge \sup_{\xis \in \Xi^{res}}
\MMn(\xi) -A_n,
\end{equation}
where $\MMn$ is defined  in (\ref{MyMn}) and   $A_n$ is a sequence of real numbers with $A_n\to 0$.
%, and $\widehat{\xi}_n$ exists thanks to Proposition~\ref{consistency-new}.
As in the proof of Proposition~\ref{consistency-new}, define $\zeta_n(\xi)=M_n(\xi)- M_n(\xi_0)$ to obtain that, for $\epsilon_0$ small enough,

\begin{equation}
  \sup_{\Vert \xis-\xis_0\Vert =\varepsilon_0}  \zeta_n (\xi)  \le \frac{1}{2}  \sup_{\Vert
\xis-\xis_0\Vert =\varepsilon_0 } \left\{ M(\xi)-M(\xi_0)\right\} := - \delta (\varepsilon_0) \label{doss}
\end{equation}
for $n$ large enough. Under Condition \ref{C1}, $\xi_0\in \Xi^{\res}$, and therefore,  by
% Proposition~\ref{consistency-new} and
 \eqref{deter},
\begin{equation}
\label{tress} \zeta_n(\xires) = M_n (\xires) - M_n (\xi_0) \ge
M_n (\xires)- \sup_{\xis \in \Xi^{\res}  } M_n (\xi ) \ge  -
A_n.
\end{equation}
Since $A_n \rightarrow 0$,  $-A_n > -\delta(\varepsilon_0)$ for $n $ large enough.
Combining this with (\ref{doss}) and (\ref{tress}) obtains
\[
\sup_{\Vert \xis-\xis_0\Vert =\varepsilon_0}  \zeta_n(\xi) <\zeta_n(\xires).
\]
We will deduce that $\Vert \xires -\xi_0\Vert <\varepsilon_0$, once we prove  that
\begin{equation}\label{result}
\sup_{\Vert \xis-\xis_0\Vert=\varepsilon_0} \zeta_n(\xi)=\sup_{\Vert \xis-\xis_0\Vert\geq\varepsilon_0} \zeta_n(\xi).
 \end{equation}
Now, let  prove \eqref{result}. Choose $\xi$ with $\Vert \xi -\xi_0\Vert >\varepsilon_0$, and take $t\in (0,1)$ such that  $\widetilde \xi = (1-t)\widehat\xi_n+t\xi$ is a distance $\varepsilon_0$ from $\xi_0$, where $\widehat\xi_n$ is the maximizer of $\zeta_n$ over $\Xi$, as defined in Proposition~\ref{consistency-new}, which is assumed to be at distance smaller than $\varepsilon_0$ from $\xi_0$. Then,
\[
\zeta_n(\widetilde\xi)=\zeta_n\left((1-t)\widehat \xi_n+t\xi\right)\geq
(1-t)\zeta_n(\widehat \xi_n)+t\zeta_n(\xi)\geq \zeta_n(\xi).
\]
Thus,
\[
\sup_{\Vert \xis-\xis_0\Vert =\varepsilon_0} \zeta_n(\xi)\geq \zeta_n(\widetilde \xi)
\geq \zeta_n(\xi)\;,\quad\mbox{for any $\xi$ with $\Vert \xi -\xi_0\Vert >\varepsilon_0$}
\]
which in turn yields~\eqref{result}.
%\[\sup_{\Vert \xis-\xis_0\Vert =\varepsilon_0} \zeta_n(\xi)\geq\sup_{\Vert \xis-\xis_0\Vert \geq\varepsilon_0} \zeta_n(\xi).\]
When $A_n=o_p(1)$,  convergence in probability of $\{\xires\}$ to $\xi_0$ is equivalent to the existence of an almost everywhere convergent  sub-sub-sequence for any subsequence $\{\widehat\xi^{\res}_{n_k}\}$. Therefore, by applying  the deterministic result to the set of probability one, where there exists a sub-subsequence of $A_{n_k}$ that converges to zero a.e., we obtain the result.
\end{proof}

\medskip

\noindent
\textbf{Regularity conditions for Proposition \ref{first} (from Theorem 5.23, p. 53 in~\cite{vander})}

\begin{condition}\label{(a)}
For each $\xi$ in an open subset $\Xi$ of a Euclidean space,  $m_{\xis} (z) $ is  a measurable function in $z$ such that $m_{\xis} (z)$
is differentiable in $\xi_0$ for almost every $z$ with derivative ${\dot m}_{\xis_0}(z)$.
\end{condition}
\begin{condition}\label{(b)}
There exists a measurable function $\phi(z)$ with $P \phi^2 < \infty$ such that,  for any $\xi_1 $ and $\xi_2 $ in a neighborhood of $\xi_0$,
\begin{equation}
\label{lips-teo}
  |m_{\xis_1} (z) - m_{\xis_2} (z) | \le  \phi(z) \| \xi_1 - \xi_2\|
\end{equation}
\end{condition}
\begin{condition}\label{(c)}
The map $\xi \to P m_{\xis} $ admits a second
order Taylor expansion at a point of maximum $\xi_0$ with nonsingular
symmetric second derivative matrix $V_{\xis_0}$.
\end{condition}

Under regularity conditions~\ref{(a)},~\ref{(b)} and~\ref{(c)},  van der Vaart proved in Theorem 5.23 of his book (\cite{vander})  that
if $\{\widehat{\xi}_n \}$ is a strong M-estimator sequence  for the criterion function $\mathbb P_nm_{\xis}$ over  $\Xi$ and
$\widehat{\xi}_n \rightarrow \xi_0$ in probability, then
\begin{align}\label{infl0}
\sqrt{n} (\widehat{\xi}_n - \xi_0) &=
- V_{\xis_0}^{-1} \frac{1}{\sqrt{n}} \sum_{i=1}^n {\dot
m}_{\xis_0}  (Z_i)+ o_p(1).
\end{align}
Moreover,
$\sqrt{n} (\widehat{\xi}_n -
\xi_0)$ is asymptotically normal with mean zero
and
\begin{equation}
 \label{general-avar}
\mbox{avar}\left\{\sqrt{n} (\widehat{\xi}_n -
\xi_0)\right\}=V_{\xis_0}^{-1}P{\dot m}_{\xis_0}{\dot
m}_{\xis_0}^{T} V_{\xis_0}^{-1}.
 \end{equation}
This result will be invoked in the following proofs.

\begin{proof}[Proof of Proposition ~\ref{first}]
Assume  that  $\{\xires\}$ is a sequence in $\Xi^{\res}$ that converges in probability to $\xi_0\in \mathcal M$, which is assumed to be open in $\Xi^{\res}$. Then, $\pr (\xires\in \mathcal M)\to 1$. Bicontinuity of $h$ guarantees that  $s^\ast_n=h^{-1}(\xires)$ converges in probability to $s_0=h^{-1}(\xi_0)$. Note that
\begin{equation}
 P_nm_{h(s^\ast_n)}=  P_nm_{\sxires} \geq \sup_{\xis^{\scriptstyle}\in \Xi^{\res} }P_n m_{\xis}
 \geq \sup_{\xis^{\scriptstyle }\in \mathcal M }P_n m_{\xis} \geq \sup_{s\in \mathcal S } P_n  m_{h(s)},
\end{equation}
except for an  $o_p(n^{-1})$ term that is omitted  in the last three inequalities. Therefore, $\{s^\ast_n\}$ is a strong maximizing sequence for the criterion function $P_nm_{h(s)}(z)$ over $\mathcal S$.

We next verify Conditions~\ref{(a)},~\ref{(b)} and~\ref{(c)} %of Theorem 5.23
are satisfied for $\{ s_n^{*} \}$, $s_0$,  $m_{h(s)}(z)$ and $Pm_{h(s)}$.
Specifically, Condition \ref{(a)} holds since $m_{h(s)}$ is a measurable function in $z$ for all $s \in
\mathcal{S}$ and $m_{h(s)}(z)$ is differentiable
in $s_0$ for almost every $z$. In fact, $h(s_0) = \xi_0$, $m_\xi(z)$
is differentiable at $\xi_0$ and $h(s)$ is also differentiable. Moreover, the derivative function is $\nabla h(s_0){\dot
m}_{\xis_0}$.
%In fact, $h(s) \in \mathcal{M}$ and this condition is valid in the unconstrained problem.

%\item
For all $s_1$ and $s_2$ in a neighborhood of $s_0$, by the continuity of $h$, $h(s_1)$ and $h(s_2)$ are in a neighborhood of $\xi_0$.
Then
\begin{eqnarray*}
| m_{h(s_1)}(z)-m_{h(s_2)}(z)|&\leq& \phi(z) \| h(s_1)-h(s_2)
\|
\\
&\leq & \phi(z)\|\nabla h\|_{\infty,\mathcal N_{ s_0}} \|s_1-s_2 \|
\end{eqnarray*}
where $\|\nabla h \|_{\infty,\mathcal N_{ s_0}}$ denotes the maximum of $\|\nabla h(s) \|$ in a neighborhood $\mathcal N_{ s_0}$ of $s_0$. The first inequality holds because  such condition is valid in
the unconstrained problem and the second inequality follows since $h$ is continuously differentiable at $s_0$. Thus, the Lipschitz Condition \ref{(b)} is satisfied.

For Condition~\ref{(c)}, we observe that the function $s \mapsto P m_{h(s)}$ is twice continuously differentiable in $s_0$ because both  $P m_{\xis}$ and $h(s)$ satisfy the required regularity properties at $\xi_0$ and $s_0$, respectively. Moreover,  since $P\dot m_{\xis_0} =0$, the second derivative matrix  of $Pm_{h(s)}$ at $s_0$, is $W_{s_0}= \nabla h(s_0)^T V_{\xis_0}
\nabla h(s_0)$, where $V_{\xis_0}$ is the second derivative matrix of $P m_{\xis}$ at $\xi_0$. The matrix $W_{s_0}$ is nonsingular and symmetric because $\nabla h (s_0)$ is full rank and $V_{\xis_0}$ is nonsingular and symmetric. % 

We can now apply Theorem 5.23 in van der Vaart (2000), and obtain
\begin{equation}%\label{aprox lineal de sn}
\sqrt{n}(s_n^{*} - s_0)= -\left(
\nabla h(s_0)^T V_{\xis_0}\nabla h (s_0)\right)^{-1}
\frac{1}{\sqrt{n}} \sum_{i=1}^n \nabla h(s_0)^T{\dot m}_{\xis_0}(Z_i)+ o_p(1),
\end{equation}
so that the first order Taylor series expansion of $h(s_n^{*})$  around $s_0$ is
\begin{eqnarray*}
\sqrt{n}(h(s_n^{*})- h(s_0))&=& \nabla h(s_0)\sqrt{n}(s_n^{*} - s_0) )+o_p(1) \\
                               &=& -\frac{1}{\sqrt n} \sum_{i=1}^ n
                                                          \nabla h(s_0)
                               \left(\nabla h(s_0)^T V_{\xis_0}\nabla h(s_0)\right)^{-1}
\nabla h(s_0)^T{\dot m}_{\xis_0}(Z_i)+o_p(1)\\
    &=& \frac{1}{\sqrt n} \sum_{i=1}^ n \Pi_{\xis_0(-V_{\xi_0}} IF_{\xis_0} (Z_i)+o_p(1),\\
\end{eqnarray*}
for $\Pi_{\xis_0(-V_{\xi_0}}$  and  $IF_{\xis_0} (z)$ defined in  (\ref{Pi General}) and (\ref{infl1}), respectively, which obtains the expansion in (\ref{infl2}).
\end{proof}

\medskip

\begin{proof}[Proof of Theorem~\ref{teorema GLM}]
%The calculations required to arrive at the asymptotic variance expression in (\ref{W}),
%necessitates the use of the following notation
Write
\begin{eqnarray*}
\eta_{x 1}= \overline{\eta}_1+ \beta x
=\left(\overline{\eta}_1, \beta \right)f(x) = (f(x)^T
\otimes I_{k_1})\vecc\left(\overline{\eta}_1, \beta\right),
\end{eqnarray*}
where {{$f(x)^T= (1, x^T)\in \real^{ 1\times (p+1)}$}}. Then, in matrix form,
\begin{equation} \label{ModelMatricialGLM}
\eta_{x}=\left( \begin{array}{c}
             \overline{\eta}_1+ \beta x \\
     \overline{\eta}_2
             \end{array}\right)=\left(
                                 \begin{array}{cc}
                                   f(x) ^T \otimes I_{k_1} & 0 \\
                                   0 & I_{k_2} \\
                                 \end{array}
                               \right) \left(
                                         \begin{array}{c}
                                           \vecc\left(\overline{\eta}_1, \beta\right) \\
                                            \overline{\eta}_2\\
                                         \end{array}
                                       \right)=
                                       F(x)\xi,
\end{equation}
where
\[F(x)=\left(\begin{array}{cc}
                     f(x) ^T \otimes I_{k_1} & 0 \\
                      0 & I_{k_2} \\
                     \end{array} \right) \in \real^{k \times (k_1 (p+1)+k_2) },
\]
and $\xi=\left(\overline{\eta}_1^T,\vecc^T(\beta),\overline{\eta}_2^T \right)^{T}$ is the vector of parameters of model (\ref{modelo1}). Note that $F(x)\xi \in H$  for any value of $\xi$ {{with $H$ defined in (\ref{natural-space})}}. This  notation allows to simplify the expression for the log-likelihood function in (\ref{m de glm}), and replace it with
\begin{equation}
 \label{m-new-notation}
m_{\xis}(z)=T(y)^{T}F(x)\xi - \psi(F(x)\xi).
\end{equation}
The regularity conditions required to derive consistency and asymptotic distribution of the MLE are:
For any $\xi$ and any $\overline \eta$ and for any compact $\mathcal K\subset \Xi=\real^{k_1}\times \real^{k_1p}\times H_2$.
\begin{equation}\label{concentrado}
 \pr \left( F(X)\xi =\overline \eta\right)<1.
\end{equation}
\begin{equation}\label{forprop1}
%  E\left[\Vert F(\X)\Vert    \right]\; <\; \infty\;,\quad
  \E\left[\Vert  T(Y)^ T F(X)\Vert  \right]\; <\; \infty\;,\quad
 \E\left[\sup_{\xi \in \mathcal K}\vert \psi(F(X)\xi)\vert \right] \; <\; \infty,
\end{equation}
\begin{eqnarray}
&& \label{nabla-psi-bounded} \E\left[\Vert  T(Y)^ T F(X)\Vert^2 \right]\; <\; \infty\;,\quad
 \E\left[\sup_{\xi \in \mathcal K}\Vert\nabla \psi(F(X) \xi) F(X)\Vert^2 \right] \;<\;\infty\;,\quad \\
 &&\label{hessiano} \E\left[\sup_{\xi \in \mathcal K}\Vert F(X)^T\nabla^2
 \psi(F(X)\xi) F(X)\Vert \right] \;<\;\infty .
\end{eqnarray}
%for any compact $\mathcal K\subset \Xi=\real^{k_1}\times \real^{k_1p}\times  H_2$.

To prove the existence, uniqueness and consistency of the MLE under the present model,  $\widehat\xi_n$,
we  next show that the assumptions stated in Proposition \ref{consistency-new} are satisfied.

The strict convexity of $\psi$ implies that,  for each fixed $z$, $m_{\xis}(z)$ is a strictly concave function  in $\xi \in \Xi=\real^{k_1+{{p}} k_1}\times H_2 $.
The concavity of $m_\xis(z)$ is preserved under expectation, thus $M(\xi)=Pm_\xis$ is  concave. The identifiability condition satisfied by  the exponential family in Section \ref{exponential-family} allows applying Lemma~5.35 of van der Vaart (2000, p. 62) to conclude that
\begin{equation}
  \E\left[T(Y)^{T}F(Y)\xi - \psi(F(X)\xi)\mid X\right]
  \;\leq\;
   \E\left[T(Y)^{T}F(X)\xi_0 - \psi(F(X)\xi_0)\mid X \right]
\end{equation}
Taking expectation with respect to $X$, we conclude that $Pm_\xis\leq Pm_{\xis_0}$ for any $\xi$.
 Moreover, if $Pm_{\xis_1}= Pm_{\xis_0}$, $\pr (F(X)\xi_1=F(X)\xi_0)=1$, which contradicts the hypothesis (\ref{concentrado}). Finally,  the integrability of (\ref{envelope}) follows from (\ref{forprop1}).

The conditions~\ref{(a)},~\ref{(b)} and~\ref{(c)} required by van der Vaart's Theorem 5.23 to derive the asymptotic distribution of M-estimators are easily verifiable under the integrability assumptions stated in (\ref{nabla-psi-bounded}) and (\ref{hessiano}).

The second derivative matrix of $Pm_{\xis}$ at $\xi_0$ is
\begin{equation}
 V_{\xis_0}=-\E\left[F(X)^ T\nabla^ 2\psi(F(X)\xi_0)F(X)\right].
\end{equation}
Finally, observe  that
\begin{eqnarray*}
  P\left[{\dot m}_{\xis_0}{\dot m}_{\xis_0}^{T}\right]&=&
  \E\left[ F(X)^T \left\{T(Y)-\nabla^T \psi(F(X)\xi_0)\right\}
  \left\{T(Y)^T-\nabla\psi(F(X)\xi_0) \right\} F(X)\right]\\
  &=& \E\left\{ F(X)^T \nabla^ 2 \psi(F(X)\xi_0)F(X)\right\}
  \end{eqnarray*}to deduce  that, according to the general formula for  the asymptotic variance of an M-estimator in (\ref{general-avar}), the asymptotic variance of $\sqrt{n}(\widehat\xi_n-\xi_0)$ is given by (\ref{W}).
\end{proof}

Lemmas~\ref{de beta0}--~\ref{M abierto} and Corollary~\ref{preimagen} are required to prove Proposition \ref{condiciones de RR}.

\begin{lemma}\label{de beta0}
Assume that $\beta_{01} \in  \real^{k_1\times r}_{\firsts,d}$ and can be written as in (\ref{beta 0}). Let $(S_0, T_0) \in \real^{k_1\times d}_d\times \real^{d\times q}_d$ with $S_0 T_0=\beta_{01}$. Then, there exists $U\in \real^{d\times d}_d$ so that $S_0=S_0(U)$ and $T_0=T_0(U)$, with
\begin{eqnarray}\label{deU}
S_0(U):= \left( \begin{array}{c} U \\ A_0 U
\end{array} \right) \quad\mbox{and}\quad T_0(U):= U^{-1}B_0.
\end{eqnarray}
\end{lemma}

\begin{proof}
Let
\[S_0 =  \left( \begin{array}{c} S_{01} \\ S_{02}
\end{array} \right)\]
with $S_{01} \in {\mathbb R}^{d \times d}$ and $S_{02}\in {\mathbb R}^{(k_1-d)\times d}$. The matrix $S_{01}T_0$ is comprised of the first $d$ rows of $\beta_{01}$ which  are linearly independent. Then, $d=\rank(S_{01} T_0)\leq \rank(S_{01})\leq d$, hence $S_{01}$ is invertible.
Take $U=S_{01}$. From the expression (\ref{beta 0}) for $\beta_{01}$, we have
\begin{eqnarray*}
S_{01} T_0&=& B_0 \\
S_{02} T_0&=&A_0 B_0.
\end{eqnarray*}
Thus, $T_0=U^{-1}B_0$, and since $T_0 T_0^T \in \mathbb{R}^{d\times d}$ and rank $d$,
\begin{eqnarray*}
S_{02} &=&A_0 B_0 T_0^T(T_0 T_0^T)^{-1}=A_0 B_0 B_0^{T} U^{-1}( U^{-1} B_0 B_0^{T} U^{-1})^{-1}=A_0 U.
\end{eqnarray*}
\end{proof}

\medskip

\begin{corollary}\label{preimagen}
Let $\beta_{01} \in  \real^{k_1\times r}_{\firsts,d}$ and can be written as in (\ref{beta 0}) . For $U$ in $\real^{d\times d}_d$ and $S_0(U), T_0(U)$ as defined in (\ref{deU}), the pre-image of $\xi_0$ through $g$, $g^{-1}(\xi_0)$, satisfies
\begin{equation}
 g^{-1}(\xi_0)\cong \left\{ \left( \overline\eta_{01}, S_0(U), T_0(U),
 \beta_{02},\overline\eta_{02} \right): U\in \real^{d\times d}_d\right\}.
\end{equation}
\end{corollary}

\medskip

\begin{lemma}\label{abierto}
$\real^{d\times m}_d$ {{with $d\le m $}} is an open set in $\real^{d\times m}$.
 \end{lemma}
\begin{proof}
We will show that the complement of  $\real^{d\times m}_d$ is closed. Consider $(T_n)_{n\geq 1} \in \real^{d\times m}$, each
$T_n$ of rank strictly less than $d$, and assume that $T_n$ converges to $T \in \real^{d\times m}$. Note that, $\vert T_n T_n^T\vert=0$ for all $n\geq 1$, so that $\vert TT^T\vert$ is also equal to zero. Hence, $\rank(T)<d$.
\end{proof}

\medskip

\begin{lemma}\label{abierto Theta y S}
 $\Theta$ and $\mathcal S$ are open sets.
\end{lemma}
\begin{proof}
 Up to an homeomorphism, $\Theta$  and $\mathcal S$ are equivalent to
 \[ \real^{k_1} \times \real^{k_1\times d}_d \times \real^{d\times { {r}}}_d\times \real^{k_1\times (p-{ {r}})}
 \times
 H_2\quad \mbox{and}\quad \real^{k_1}\times \real^{(k_1-d)\times d} \times \real^{d\times {{r}}}_d\times \real^{k_1\times (p-{{r}})} \times H_2,\]
 respectively, which are products of open sets by Lemma \ref{abierto}.
\end{proof}

\medskip

\begin{lemma}\label{p-bicont}
 $h: \mathcal S\to \mathcal M$ is one to one  bicontinuous.
\end{lemma}
\begin{proof}
It suffices to note that $h_1: \real^{(k_1-d)\times d}\times \real^{d\times q}_d \mapsto \real^{k_1\times q}_{\firsts,d}$ with
 \begin{eqnarray*}
 (A, B) \to \left( \begin{array}{c} B \\ A B
\end{array} \right)
\end{eqnarray*}
satisfies the required properties for $h$. Let  $h_1^{-1}:\real^{k_1\times
q}_{\firsts,d} \mapsto \real^{(k_1-d)\times d}\times \real^{d\times q}_d$  with
\begin{equation}
  \left(
\begin{array}{c} B \\ C \end{array} \right)      \to (C B^T\left(BB^{T}\right)^{-1} ,B)
\end{equation}
Then, $h_1^{-1} ( h_1 (A, B) ) = (A, B)$ and \[h_1 \left(h_1^{-1}   \left(
\begin{array}{c} B \\ A B \end{array} \right)   \right)  =   \left( \begin{array}{c} B \\  AB  \end{array} \right)   ,\] and therefore $h_1^{-1}$ is the inverse of  $h_1$ .
Thus, $h_1$ is one to one and bicontinuous.
\end{proof}

\medskip

\begin{lemma}\label{M abierto}
 $\mathcal M$ is open in $\Xi^{\res}$
 \end{lemma}

\begin{proof}
It suffices to show $\mathbb{R}^{k_1 \times r}_{\firsts,d}$ is an open set in $\mathbb{R}^{k_1 \times r}_{d}$, or equivalently, that the complement of  $\mathbb{R}^{k_1 \times r}_{\firsts,d}$ is a closed set in $\mathbb{R}^{k_1 \times r}_{d}$.
Let $\{T_n\} \subset \mathbb{R}^{k_1 \times r}_{d}$ be a sequence such that the first $d$ rows are not linearly independent and $T_n$ converges to $T \in \mathbb{R}^{k_1 \times r}_{d}$. Then, if we write \[T_n = \left(
                \begin{array}{c}
                  T_{n_1} \\
                  T_{n_2} \\
                \end{array}
              \right) \]
with $T_{n_1} \in \mathbb{R}^{d \times r}$ and $T_{n_2} \in \mathbb{R}^{(k_1-d) \times r}$, we obtain that $\mid T_{n_1}T_{n_1}^{T}\mid=0$ for all $n$. Then $\mid T_{1}T_{1}^{T}\mid=0$, where $T_{1}$ comprises of the first $d$ rows of $T$, and therefore the first $d$ rows of $T$ are not linearly independent.
\end{proof}

\medskip

\begin{proof}[Proof of Proposition \ref{condiciones de RR}]
We verify that Conditions \ref{C1}, \ref{C2} and
\ref{C3} are satisfied.

Condition \ref{C1}: By Lemma \ref{abierto Theta y S}, the set $\Theta$, defined in (\ref{Theta-example}), is open and $\xi_0$ belongs to $g(\Theta)=\Xi^{\res}$.

Condition \ref{C2}:
Since the first $d$ rows of $\beta_{01}$ are linearly independent, $\xi_0 \in \mathcal{M}$.  Then, applying Lemma \ref{M abierto}  obtains that $\mathcal M$  is an open set in $\Xi^{\res}$.

Consider  $(\mathcal{S},h)$ as in (\ref{S-example}).  By Lemma \ref{abierto Theta y S}, $\mathcal{S}$ is an open set, and  $h$ is one-to-one and bicontinuous by Lemma \ref{p-bicont}. Furthermore, if
\[s_0 =(\overline{\eta}_{01}, A_0, B_0, \beta_{02}, \overline{\eta}_{02})\in \mathcal{S},\]
where $A_0$ y $B_0$ are given in (\ref{beta 0}), then $h(s_0)=\left(\overline{\eta}_{01},\ \beta_{01}, \beta_{02},\overline{\eta}_{02}\right)=\xi_0$.

The function $h$ is twice continuously differentiable and its Jacobian  is full rank. In fact, the latter is
\begin{equation}
\nabla h(\overline{\eta}_1, A,B,\beta_2,\overline{\eta}_2)=
\left(
    \begin{array}{ccccc}
    I_{k_1}&0 & 0 & 0 & 0\\
    0& B^{T} \otimes \left(
                       \begin{array}{c}
                         0 \\
                         I_{k_1-d} \\
                       \end{array}
                     \right)
      &  I_{{ {r}}} \otimes \left(
                        \begin{array}{c}
                          I_d \\
                          A \\
                        \end{array}
                      \right) & 0 &0
      \\
      0& 0 &0& I_{k_1(p-{ {r}})}&0\\
      0&0&0&0&I_{k_2}
    \end{array}
  \right)
\end{equation}
of order $(k_1p+k) \times (d(k_1+{ {r}}-d)+k_1(p-{ {r}})+k)$ with full column rank $d(k_1+{ {r}}-d) + k_1(p-{ {r}}) + k$ [see \cite{CookNi2005}; also it is
a direct implication of Theorem 5 in \cite{MatrixTricks}].

Condition \ref{C3}: Consider   $\theta_0\in g^{-1}(\xi_0)$, associated with $U$, as shown in Lemma \ref{de beta0}.
%We will show that  $\spn \nabla g(\theta_0)=\spn \nabla h(s_0)$.
Since
\[ \nabla g(\overline{\eta}_1, S, T, \beta_2,\overline{\eta}_2) =\left( \begin{array}{ccccc}
I_{k_1} &0 &0 &0 &0\\
0&T^T \otimes I_{k_1}  & I_{ {r}} \otimes S &  0 &0 \\
      0&0& 0 & I_{k_1(p-{ {r}})}&0\\
      0& 0 & 0 &0& I_{k_2}
      \end{array}\right) \]
then,
\begin{eqnarray*}
&&\nabla g \left(\overline{\eta}_{10},  \left(\begin{array}{c}
                    U \\
                    A_0 U \\
                    \end{array}
                   \right)
,  U^{-1}B_0, \beta_{20},\overline{\eta}_{20} \right)\\
&& = \left( \begin{array}{ccccc}
                  I_{k_1} &0 &0 &0 &0\\
                  0& B_0^T \otimes I_{k_1}  & I_{ {r}} \otimes  \left(\begin{array}{c}
                    I_d \\
                    A_0 \\
                    \end{array}
                   \right)&  0 &0 \\
0&      0& 0 & I_{k_1(p-{ {r}})}&0\\
      0& 0 & 0 & 0&I_{k_2}
                   \end{array}\right)G.
       \end{eqnarray*}
where
\begin{equation}
G= \left(
     \begin{array}{ccccc}
        I_{k_1} &0 &0 &0&0\\
         0&U^{-T}\otimes I_{k_1}  & 0 & 0 &0\\
         0& 0 &  (I_{ {r}}\otimes U ) &0&0\\
         0 & 0&0&
         I_{k_1(p-{ {r}})}&0\\
         0&0&0&0&I_{k_2}
        \end{array}\right)
\end{equation}
Since $G$  is invertible of order $(d(k_1+{ {r}})+k_1(p-{ {r}})+k) \times
(d(k_1+{ {r}})+k_1(p-{ {r}})+k)$,
\begin{eqnarray*}
\spn \nabla g(\theta_0)&=&\spn \left( \begin{array}{ccccc}
                  I_{k_1} &0 &0 &0 &0\\
                  0& B_0^T \otimes I_{k_1}  & I_{ {r}} \otimes  \left(\begin{array}{c}
                    I_d \\
                    A_0 \\
                    \end{array}
                   \right)&  0 &0 \\
0&      0& 0 & I_{k_1(p-{ {r}})}&0\\
      0& 0 & 0 & 0&I_{k_2}
                   \end{array}\right).
\end{eqnarray*}
Next, since  {{$\nabla h (s_0)= \nabla g (\theta_0) G^{-1} \widetilde{I} $,}} with
 \[\widetilde{I}=\diag\left(
                                     \begin{array}{ccccc}
                                      I_{k_1}& I_d \otimes \left(
                                                      \begin{array}{c}
                                                        0 \\
                                                        I_{k_1-d} \\
                                                      \end{array}
                                                    \right)
                                        & I_{{ {r}}d} & I_{k_1(p-{ {r}})} & I_{k_2} \\
                                     \end{array}
                                   \right)\]
we have $\spn\nabla h(s_0) \subset \spn \nabla g(\theta_0)$.
Applying again Theorem 5 in \cite{MatrixTricks} obtains
$\rank(\nabla g(\theta_0))=d(k_1+{ {r}}-d)+ k_1(p-{ {r}}) +k=\rank(\nabla h(s_0))$. Therefore, $\spn\nabla h(s_0) = \spn \nabla g(\theta_0)$.
\end{proof}

\medskip

\begin{proof}[Proof of Theorem~\ref{teorema GLM RR}]
In Theorem \ref{teorema GLM} we verified  that the conditions of Theorem 5.23 in \cite{vander} are satisfied for the maximum likelihood estimator under model (\ref{expfamily}) satisfying  (\ref{modelo1}). The asymptotic variance of the MLE estimator is given in~(\ref{W}). We also showed Conditions~\ref{C1}, \ref{C2} and  \ref{C3}  are satisfied in the proof of Proposition \ref{condiciones de RR}. The result follows from Proposition \ref{first} since $-V_{\xi_0}=W_{\xi_0}^{-1}$ and
\begin{eqnarray*}
  \hbox{avar}\{\sqrt{n} (\xires - \xi_0)\} &=& \Pi_{\xis_0(W_{\xis_0}^{-1})}W_{\xis_0} \Pi_{\xis_0(W_{\xis_0}^{-1})}^T\\
                                              &=& \nabla g(\theta_0)(\nabla g(\theta_0)^TW_{\xis_0}^{-1}
\nabla g(\theta_0))^
 {\dag}\nabla g(\theta_0)^T\\
 &=&\Pi_{\xis_0(W_{\xis_0}^{-1})}W_{\xis_0}
\end{eqnarray*}

\end{proof}

\end{document}